\documentclass[12pt]{amsart}

\usepackage{amscd}
\usepackage{amsthm}
\usepackage{amsmath}
\usepackage{amsfonts}
\usepackage{amssymb}
\usepackage[all]{xy}
\usepackage{graphicx}
\usepackage{verbatim}
\usepackage{color}

\newtheorem{theorem}{Theorem}[section]
\newtheorem{lemma}[theorem]{Lemma}
\newtheorem{proposition}[theorem]{Proposition}
\newtheorem{corollary}[theorem]{Corollary}

\theoremstyle{definition}
\newtheorem{definition}[theorem]{Definition}
\newtheorem{example}[theorem]{Example}

\newtheorem{remark}[theorem]{Remark}

\newcommand{\func}[1]{\operatorname{#1}}

\def\Id{{\rm Id}}

\def\Min{{\rm Min}}

\def\Hom{{\rm Hom}}
\def\ann{\func{ann}}

\newcounter{num}

\begin{document}

\title{Idempotent generated algebras and Boolean powers of commutative rings}
\author{G.~Bezhanishvili, V.~Marra, P.~J.~Morandi, B.~Olberding}

\date{}

\begin{abstract}
A Boolean power $S$ of a commutative ring $R$ has the structure of a commutative $R$-algebra, and with respect to this structure, each element of $S$ can be written uniquely as an $R$-linear combination of orthogonal idempotents so that the sum of the idempotents is $1$ and their coefficients are distinct. In order to formalize this decomposition property, we introduce the concept of a Specker $R$-algebra, and we prove that the Boolean powers of $R$ are up to isomorphism precisely the Specker $R$-algebras. We also show that these algebras are characterized in terms of a functorial construction having roots in the work of Bergman and Rota. When $R$ is indecomposable, we prove that $S$ is a Specker $R$-algebra iff $S$ is a projective $R$-module, thus strengthening a theorem of Bergman, and when $R$ is a domain, we show that $S$ is a Specker $R$-algebra iff $S$ is a torsion-free $R$-module.

For an indecomposable $R$, we prove that the category of Specker $R$-algebras is equivalent to the category of Boolean algebras, and hence is dually equivalent to the category  of Stone spaces. In addition, when $R$ is a domain, we show that the category of Baer Specker $R$-algebras is equivalent to the category of complete Boolean algebras, and hence is dually equivalent to the category of extremally disconnected compact Hausdorff spaces.

For a totally ordered $R$, we prove that there is a unique partial order on a Specker $R$-algebra $S$ for which it is an $f$-algebra over $R$, and show that $S$ is equivalent to the $R$-algebra of piecewise constant continuous functions from a Stone space $X$ to $R$ equipped with the interval topology.
\end{abstract}


\maketitle

\section{Introduction}

For a commutative ring $R$ and a Boolean algebra $B$, the Boolean power of $R$ by $B$ is the $R$-algebra $C(X,R_{\func{disc}})$ of continuous functions from the Stone space $X$ of $B$ to the discrete space $R$ (see, e.g., \cite{BN80} or \cite{BS81}). Each element of a  Boolean power of $R$ can be written uniquely as an $R$-linear combination of orthogonal idempotents so that the sum of the idempotents is $1$ and their coefficients are distinct. In this note we formalize this decomposition property by introducing the class of Specker $R$-algebras. We prove that an $R$-algebra $S$ is isomorphic to a Boolean power of $R$ iff $S$ is a Specker $R$-algebra, and we characterize Specker $R$-algebras (hence Boolean powers of $R$) in several other ways for various choices of the commutative ring $R$, such as when $R$ is indecomposable, an integral domain, or   totally ordered.

Our terminology is motivated by Conrad's concept of a Specker $\ell$-group. We recall \cite[Sec.~4.7]{Con74} that an element $g > 0$ of an $\ell$-group $G$ is singular if $h \wedge (g-h) = 0$ for all $h\in G$ with $0 \le h \le g$, and that $G$ is a Specker $\ell$-group if it is generated by its singular elements. Conrad proved in \cite[Sec.~4.7]{Con74} that a Specker $\ell$-group admits a unique multiplication such that $gh = g \wedge h$ for all singular elements $g,h$. Under this multiplication, the singular elements become idempotents, and hence a Specker $\ell$-group with strong order unit, when viewed as a ring, is generated as a ${\mathbb{Z}}$-algebra by its idempotents. Moreover, it is a torsion-free ${\mathbb{Z}}$-algebra, and hence its elements admit a unique orthogonal decomposition. Our definition of a Specker $R$-algebra extracts these key features of Specker $\ell$-groups.

For a commutative ring $R$, we give several equivalent characterizations for a commutative $R$-algebra to be a Specker $R$-algebra. One of these characterizations produces a functor from the category {\bf BA} of Boolean algebras to the category ${\bf Sp}_R$ of Specker $R$-algebras. This functor has its roots in the work of Bergman \cite{Ber72} and Rota \cite{Rot73}. We show this functor is left adjoint to the functor that sends a Specker $R$-algebra to its Boolean algebra of idempotents. We prove that the ring $R$ is indecomposable iff these functors establish an equivalence of ${\bf Sp}_R$ and {\bf BA}. It follows then from Stone duality that when $R$ is indecomposable, ${\bf Sp}_R$ is dually equivalent to the category {\bf Stone} of Stone spaces (zero-dimensional compact Hausdorff spaces). Hence, when $R$ is indecomposable, Specker $R$-algebras are algebraic counterparts of Stone spaces in the category of commutative $R$-algebras.

It follows from the work of Bergman \cite{Ber72} that every Specker $R$-algebra is a free $R$-module. For an indecomposable $R$, we show that the converse is also true. In fact, we prove a stronger result: An idempotent generated commutative $R$-algebra $S$ (with $R$ indecomposable) is a Specker $R$-algebra iff $S$ is a projective $R$-module. A simple example shows that the assumption of indecomposability is necessary here. When $R$ is a domain, an even stronger result is true: $S$ is a Specker $R$-algebra iff $S$ is a torsion-free $R$-module. Thus, the case when $R$ is a domain provides the most direct generalization of the $\ell$-group case.

For a domain $R$, we prove that the Stone space of a Specker $R$-algebra $S$ can be described as the space of minimal prime ideals of $S$, and that a Specker $R$-algebra $S$ is injective iff $S$ is a Baer ring. This yields an equivalence between the category ${\bf BSp}_R$ of Baer Specker $R$-algebras and the category {\bf cBA} of complete Boolean algebras, and hence a dual equivalence between ${\bf BSp}_R$ and the category {\bf ED} of extremally disconnected compact Hausdorff spaces.

We conclude the article by considering the case when $R$ is a totally ordered ring. It is then automatically indecomposable. We prove that there is a unique partial order on a Specker $R$-algebra $S$ for which it is an $f$-algebra over $R$, and show that $S$ is isomorphic to the $R$-algebra of piecewise constant continuous functions from a Stone space $X$ to $R$, where $R$ is given the interval topology.  These results give a more general point of view on similar results obtained for the case $R = {\mathbb{Z}}$ by Ribenboim \cite{Rib69} and Conrad \cite{Con74}, as well as  for the case $R = {\mathbb{R}}$ as considered in \cite{BMO13a}.

\section{Specker algebras and Boolean powers of a commutative ring}\label{Specker}

All algebras considered in this article are commutative and unital, and all algebra homomorphisms are unital. Throughout $R$ will be a commutative ring with $1$. In this section we introduce Specker $R$-algebras and use them to characterize Boolean powers of $R$. A key property of Specker $R$-algebras is that their elements can be decomposed uniquely into $R$-linear combinations of idempotents so that the sum of the idempotents is $1$ and their coefficients are distinct. We begin the section by formalizing the terminology needed to make precise this decomposition property.

Let $S$ be a commutative $R$-algebra. As $S$ is a commutative ring with $1$, it is well known that the set $\Id(S)$ of idempotents of $S$ is a Boolean algebra via the operations
$$ e \vee f = e + f -ef, \ \
e \wedge f = ef, \ \
\lnot e = 1-e.$$
We call an $R$-algebra $S$ \emph{idempotent generated} if $S$ is generated as an $R$-algebra by a set of idempotents. If the idempotents belong to some Boolean subalgebra $B$ of $\Id(S)$, we say that \emph{$B$ generates $S$}. Because we are assuming $S$ is commutative, each monomial $e_1^{n_1} \cdots e_r^{n_r}$ of idempotents is equal to $e_1 \cdots e_r$, which is then equal to $e_1 \wedge \cdots \wedge e_r$, an idempotent in $S$. Therefore, since each element of $S$ is an $R$-linear combination of monomials of idempotents, each element is, in fact, an $R$-linear combination of idempotents. Thus, an idempotent generated $R$-algebra $S$ is generated as an $R$-module by its idempotents, and if $B$ generates $S$, then $B$ generates $S$ both as an $R$-algebra and as an $R$-module.

We call a set $E$ of nonzero idempotents of $S$ \emph{orthogonal} if $e\wedge f=0$ for all $e\ne f$ in $E$, and we say that $s\in S$ has an \emph{orthogonal decomposition} or that $s$ is in \emph{orthogonal form} if $s=\sum_{i=1}^n a_ie_i$ with the $e_i\in\Id(S)$ orthogonal. If, in addition, $\bigvee e_i = 1$, we call the decomposition a \emph{full orthogonal decomposition}. By possibly adding a term with a $0$ coefficient, we can turn any orthogonal decomposition into a full orthogonal decomposition.

We call a nonzero idempotent $e$ of $S$ \emph{faithful} if for each $a\in R$, whenever $ae = 0$, then $a = 0$. Let $B$ be a Boolean subalgebra of $\Id(S)$ that generates $S$. We say that $B$ is a \emph{faithful generating algebra of idempotents of $S$} if each nonzero $e\in B$ is faithful.

\begin{lemma}\label{lem:2.1}
Let $S$ be a commutative $R$-algebra and let $B$ be a Boolean subalgebra of $\Id(S)$ that generates $S$. Then each $s\in S$ can be written in full orthogonal form $s=\sum_{i=1}^n a_ie_i$, where the $a_i\in R$ are distinct and $e_i \in B$. Moreover, such a decomposition is unique iff $B$ is a faithful generating algebra of idempotents of $S$.
\end{lemma}

\begin{proof}
The proof that each $s \in S$ can be written in full orthogonal form is a standard argument: Write $s=\sum_{i=1}^n a_ie_i$ with $a_i\in R$ and $e_i\in B$. Each $e_i$ can then be refined into a sum of idempotents, each of which is a meet of a set of idempotents in $\{e_1,\ldots,e_n,1-e_1,\ldots,1-e_n\}$, in such a way that the resulting refinements of the $e_i$ are orthogonal. By combining terms with the same coefficient, $s$ can be written in orthogonal form with distinct coefficients. If the decomposition is not in full orthogonal form, adding the term $0f$, where $f$ is the negation of the join of the idempotents in the decomposition, turns it into a full orthogonal decomposition.

Suppose that each element has a unique full orthogonal decomposition and suppose that $ae = 0$ for some $a \in R$ and nonzero $e \in B$. Then since $ae = 0e$, uniqueness implies that $a = 0$, and hence $e$ is faithful.  Conversely, suppose that $B$ is a faithful generating algebra of idempotents of $S$. Let $s \in S$ and write $s = \sum_i a_i e_i = \sum_j b_j f_j$ with each sum a full orthogonal decomposition with distinct coefficients. First consider $i$ with $a_i \ne 0$. Multiplying both sides by $e_i$ yields $a_i e_i = \sum_j b_j (e_i f_j)$. Since $e_i$ is faithful and $a_i \ne 0$, there is $j$ with $e_i f_j \ne 0$. Multiplying by $f_j$ yields $a_i e_if_j = b_j e_i f_j$. Therefore, since $e_i f_j$ is faithful, $a_i = b_j$. Because the $b_j$ are distinct, there is a unique $j$ with $e_i f_j \ne 0$. Since $a_i = b_j$, we then have $a_i e_i = b_j e_i f_j = a_i e_i f_j$. Thus, $a_i (e_i \wedge \lnot f_j) = 0$, so by faithfulness, $e_i \wedge \lnot f_j = 0$, hence $e_i \le f_j$. Reversing the roles of $i$ and $j$ yields $f_j \le e_i$, so $e_i = f_j$. This implies that, after suitable renumbering, $e_i = f_i$ and $a_i = b_i$ for each $i$ with $a_i \ne 0$. If the decomposition $\sum a_i e_i$ has a zero coefficient, say $0 = a_k$, then as the decomposition is full, $e_k = \lnot(\bigvee_{i \ne k}e_i)$, which implies that the idempotent corresponding to a zero coefficient is uniquely determined. Consequently, $s$ has a unique full orthogonal decomposition.
\end{proof}

\begin{remark}\label{uniqueness}
\begin{enumerate}
\item[]
\item Orthogonal and full orthogonal decompositions will be our main technical tool. As we already pointed out, any orthogonal decomposition can be turned into a full orthogonal decomposition by possibly adding a term with a $0$ coefficient, so depending on our need, we will freely work with either orthogonal or full orthogonal decompositions. If $B$ is a faithful generating algebra of idempotents of $S$ and $s\in S$ is nonzero, then by possibly dropping a term with a $0$ coefficient, the same argument as in the proof of Lemma~\ref{lem:2.1} produces a unique orthogonal decomposition $s=\sum_{i=1}^n a_ie_i$, where the $a_i\in R$ are distinct and nonzero.

\item The same type of argument as in the proof of Lemma~\ref{lem:2.1} shows that if $e_1,\dots,e_n$ is an orthogonal set of faithful idempotents and $\sum a_i e_i = \sum b_i e_i$ for $a_i, b_i \in R$, then $a_i = b_i$ for each $i$. This holds regardless of whether the coefficients in either expression are distinct. We will use this fact several times.
\end{enumerate}
\end{remark}


\begin{definition}\label{def of Specker}
We call an $R$-algebra $S$ a \emph{Specker $R$-algebra} if $S$ is a commutative $R$-algebra that has a faithful generating algebra of idempotents.
\end{definition}

Obviously each Specker $R$-algebra is idempotent generated. Moreover, if $S$ is a Specker $R$-algebra, then $1 \in \Id(S)$ is faithful, which means the natural map $R \to S$ sending $a\in R$ to $a\cdot 1\in S$ is 1-1. Thus, $R$ is isomorphic to an $R$-subalgebra of $S$. Throughout we will freely identify $R$ with an $R$-subalgebra of $S$.

To characterize Specker $R$-algebras among idempotent generated commutative $R$-algebras, we introduce a construction that associates with each Boolean algebra $B$ an idempotent generated commutative $R$-algebra $R[B]$. This construction has its roots in the work of Bergman \cite{Ber72} and Rota \cite{Rot73}.

\begin{definition}\label{def:R[B]}
Let $B$ be a Boolean algebra. We denote by $R[B]$ the quotient ring $R[\{x_e : e\in B\}]/I_B$ of the polynomial ring over $R$ in variables indexed by the elements of $B$ modulo the ideal $I_B$ generated by the following elements, as $e,f$ range over $B$: $$x_{e\wedge f} - x_e x_f, \ \ x_{e\vee f} - (x_e + x_f - x_e x_f), \ \ x_{\lnot e} - (1-x_e), \ \ x_0.$$
\end{definition}

For $e\in B$ we set $y_e = x_e + I_B \in R[B]$. Considering the generators of $I_B$, we see that, for all $e,f\in B$: $$y_{e\wedge f}=y_e y_f, \ \ y_{e\vee f}= y_e + y_f - y_e y_f, \ \ y_{\lnot e}=1-y_e, \ \ y_0=0.$$

It is obvious that $R[B]$ is a commutative $R$-algebra. From the relations above it is also clear that $y_e$ is an idempotent of $R[B]$ for each $e\in B$. Therefore, each $s\in R[B]$ can be written as $s=\sum a_i y_{e_i}$ with $a_i\in R$ and $e_i\in B$. Thus, $R[B]$ is idempotent generated. Moreover, $i_B:B\to\Id(R[B])$ given by $i_B(e)=y_e$ is a well-defined Boolean homomorphism. The following universal mapping property is an easy consequence of the definition of $R[B]$.

\begin{lemma}\label{UMP}
Let $S$ be a commutative $R$-algebra. If $B$ is a Boolean algebra and $\sigma : B \to \Id(S)$ is a Boolean homomorphism, then there is a unique $R$-algebra homomorphism $\alpha : R[B] \to S$ satisfying $\alpha \circ i_B = \sigma$.
\end{lemma}

\begin{proof}
There is an $R$-algebra homomorphism $\gamma : R[\{ x_e : e \in B \}] \to S$ such that $\gamma(x_e) = \sigma(e)$ for each $e \in B$. Since $\sigma$ is a Boolean homomorphism, each generator of $I_B$ lies in the kernel of $\gamma$. Therefore, we get an induced $R$-algebra homomorphism $\alpha : R[B] \to S$ with $\alpha(x_e + I_B) = \sigma(e)$. Thus, $\alpha \circ i_B = \sigma$. Clearly $\alpha$ is the unique $R$-algebra homomorphism satisfying this equation since $R[B]$ is generated by the $y_e$.
\end{proof}

\begin{lemma}\label{lem:faithful}
Let $B$ be a Boolean algebra.
\begin{enumerate}
\item If $e \in B$ is nonzero, then $y_e \in R[B]$ is faithful.
\item $i_B$ is a Boolean isomorphism from $B$ onto a faithful generating algebra of idempotents $\{y_e:e\in B\}$ of $\Id(R[B])$.
\end{enumerate}
\end{lemma}

\begin{proof}
(1) Suppose $e \ne 0$. Then there is a Boolean homomorphism $\sigma$ from $B$ onto the two-element Boolean algebra {\bf 2} with $\sigma(e)=1$. Viewing {\bf 2} as a subalgebra of $\Id(R)$, we may view $\sigma$ as a Boolean homomorphism from $B$ to $\Id(R)$. Then, by Lemma~\ref{UMP}, there is an $R$-algebra homomorphism $\alpha : R[B] \to R$, which sends $y_e$ to $\sigma(e) = 1$. Consequently, if $ay_e = 0$, then $0=\alpha(ay_e) = a$. This shows that $y_e$ is faithful.

(2) It is obvious that $\{y_e:e\in B\}$ is a generating algebra of idempotents of $\Id(R[B])$ and that $i_B:B\to\{y_e:e\in B\}$ is an onto Boolean homomorphism. That $\{y_e:e\in B\}$ is faithful and so $i_B$ is 1-1 follows from (1).
\end{proof}

We are ready to prove the main result of this section, which gives several characterizations of Specker $R$-algebras, one of which is as Boolean powers of $R$.

\begin{theorem}\label{new def}
Let $S$ be a commutative $R$-algebra. The following are equivalent.
\begin{enumerate}
\item $S$ is a Specker $R$-algebra.
\item $S$ is isomorphic to $R[B]$ for some Boolean algebra $B$.
\item $S$ is isomorphic to a Boolean power of $R$.
\item  There is a Boolean subalgebra $B$ of $\Id(S)$ such that $S$ is generated by $B$ and every Boolean homomorphism $B\to{\bf 2}$ lifts to an $R$-algebra homomorphism $S\to R$.
\end{enumerate}
\end{theorem}

\begin{proof}
(1) $\Rightarrow$ (2): Let $B$ be a faithful generating algebra of idempotents of $S$. By Lemma~\ref{UMP}, the identity map $B\rightarrow  B$ lifts to an $R$-algebra homomorphism $\alpha:R[B] \rightarrow S$. By assumption, $B$ generates $S$, so $\alpha$ is onto. To see that $\alpha$ is 1-1, suppose that $s \in R[B]$ with $\alpha(s) = 0$. Since $R[B]$ is idempotent generated, by Lemma~\ref{lem:2.1}, we may write $s = \sum a_i y_{e_i}$, where the $a_i \in R$ are distinct and the $e_i\in B$ are orthogonal. Therefore, $0 = \alpha(s) = \sum a_i e_i$. Multiplying by $e_i$ gives $a_ie_i = 0$, which since the nonzero idempotents in $B$ are faithful implies that $a_i = 0$. This yields $s = 0$; hence, $\alpha$ is an isomorphism.

(2) $\Rightarrow$ (3): We show that  $R[B]$ is isomorphic to $C(X,R_{\func{disc}})$, where $X$ is the Stone space of $B$. By Stone duality, we identify $B$ with the Boolean algebra of clopen subsets of $X$. For $e$ a clopen subset of $X$, let $\chi_e$ be the characteristic function of $e$, and define $\sigma : B \to C(X, R_{\func{disc}})$ by $e \mapsto \chi_e$. It is easy to see that this is a Boolean homomorphism from $B$ to the idempotents of $C(X, R_{\func{disc}})$. Thus, by Lemma~\ref{UMP}, there is an $R$-algebra homomorphism $\alpha : R[B] \to C(X, R_{\func{disc}})$ which sends $y_e$ to $\sigma(e)$ for each $e \in B$. By Lemma~\ref{lem:2.1}, each $s \in R[B]$ may be written in the form $s = \sum a_i y_{e_i}$ with the $a_i\in R$ distinct and the $e_i\in B$ orthogonal. Then $\alpha(s)$ is the continuous function $X \to R$ such that
\[
\alpha(s)(x) =
\left\{\begin{array}{ll}
a_i & \textrm{if } x \in e_i, \\
0 & \textrm{otherwise.} \\
\end{array}\right.
\]
If $s \ne 0$, then there is $i$ with $e_i \ne \varnothing$ and $a_i \ne 0$. Therefore, $\alpha(s) \ne 0$ in $C(X, R_{\func{disc}})$. Thus, $\alpha$ is 1-1. To see $\alpha$ is onto, let $f \in C(X, R_{\func{disc}})$. For each $a \in R$, we see that $f^{-1}(a)$ is clopen in $X$, and the various $f^{-1}(a)$ cover $X$. By compactness, there are finitely many distinct $a_i$ such that $X = f^{-1}(a_1) \cup \cdots \cup f^{-1}(a_n)$. If $e_i = f^{-1}(a_i)$, then $f = \sum a_i \chi_{e_i}$, so $f = \alpha\left(\sum a_i y_{e_i}\right)$. Thus, $\alpha$ is onto. Consequently, $\alpha$ is an $R$-algebra isomorphism between $R[B]$ and $C(X,R_{\func{disc}})$.

(3) $\Rightarrow$ (1): Let $X$ be a Stone space and set $S = C(X,R_{\func{disc}})$. For each clopen subset $U$ of $X$, the characteristic function $\chi_U$ of $U$ is an idempotent of $S$. Let
\[
B = \{\chi_U : U \textrm{ is clopen in }X\}.
\]
Then $B$ is a Boolean subalgebra of $\Id(S)$. Moreover, each nonzero $\chi_U \in B$ is faithful, since if $a \in R$ with $a\chi_U = 0$, then $a\chi_U(x) = 0$ for all $x \in X$. As $\chi_U$ is nonzero, $U$ is nonempty. Let $x \in U$. Then $0 = a\chi_U(x) = a$. Thus, $\chi_U$ is faithful. Finally, we show that $B$ generates $S$. Take $s \in S$. For each $a \in R$ the pullback $s^{-1}(a)$ is a clopen subset of $X$. Moreover, $X$ is covered by the various $s^{-1}(a)$. Since $X$ is compact, there are distinct $a_1,\dots, a_n \in R$ with $X = s^{-1}(a_1) \cup \cdots \cup s^{-1}(a_n)$. If $U_i = s^{-1}(a_i)$, then $s = \sum a_i \chi_{U_i}$. Thus, $B$ generates $S$. Consequently, $S$ is a Specker $R$-algebra.

%
(2) $\Rightarrow$ (4): Suppose that $S \cong R[C]$ for some Boolean algebra $C$. Let $B = \{y_c : c \in C\}$. By Lemma~\ref{lem:faithful}, $B$ is isomorphic to $C$, so $R[B]$ is isomorphic to $R[C]$. We identify $S$ with $R[B]$. Let $\sigma : B \to {\bf 2}$ be a Boolean homomorphism. By viewing ${\bf 2}$ as a Boolean subalgebra of $\Id(R)$, we may view $\sigma$ as a Boolean homomorphism from $B$ to $\Id(R)$. Then Lemma~\ref{UMP} yields an $R$-algebra homomorphism $S \to R$ lifting $\sigma$.

(4) $\Rightarrow$ (1): It suffices to show that every nonzero idempotent in $B$ is faithful. Let $0 \ne e \in B$, and let $a \in R$ with $ae = 0$. Since $0\ne e$, there is a Boolean homomorphism $\sigma:B \rightarrow  {\bf 2}$ such that $\sigma(e) = 1$. By (4), $\sigma$ lifts to an $R$-algebra homomorphism $\alpha:S \rightarrow  R$. Thus, $0 = \alpha(ae) = a\sigma(e) = a$, so that $e$ is faithful.
\end{proof}

\begin{remark}\label{unique B}
\begin{enumerate}
\item[]
\item Let $S$ be a Specker $R$-algebra. We will see in Section~\ref{sec:3} that a faithful generating algebra of idempotents of $S$ need not be unique, but that it is unique up to isomorphism.
\item The proof of (1) $\Rightarrow$ (2) of Theorem~\ref{new def} shows that if $B$ is a faithful generating algebra of idempotents of $S$, then $S \cong R[B]$. We will make use of this fact later on.
\item In Theorem~\ref{new def}.4, the requirement that $S$ is generated by $B$ is not redundant. For, let $R$ be an integral domain and let $S = R[x]/(x^2)$. As $R$ has no zero divisors, the only $R$-algebra homomorphism from $S$ to $R$ sends the coset of $x$ to $0$. Therefore, each Boolean homomorphism ${\bf 2} \rightarrow {\bf 2}$ lifts uniquely to an $R$-algebra homomorphism  $S\rightarrow  R$. By the definition of $S$, each element of $S$ can be written uniquely as the coset of some linear polynomial $a+bx$ for $a,b \in R$. If $s = a+bx+(x^2)$ is idempotent, then $s^2 = s$ yields $a^2 + 2abx + (x^2) = a+bx+(x^2)$. Uniqueness then yields $a^2 = a$ and $2ab=b$. Therefore, $a\in\Id(R)$, and as $R$ is a domain, this forces $a=0,1$. If $a=1$, then $b=0$. Thus, $s \in \{0+(x^2),1+(x^2)\}$, and so $\Id(S) = \{0,1\}$. It follows that $S$ is not generated over $R$ by idempotents.
\end{enumerate}
\end{remark}

\begin{remark}\label{Foster remark}
While in this article we focus on viewing Boolean powers as $C(X,R_{\func{disc}})$, Foster's original conception of a Boolean power \cite{Fos53a,Fos53b,Fos61} also has an interesting interpretation in our setting. Let $R$ be a commutative ring and let $B$ be a Boolean algebra. Consider the set $R[B]^\perp$ of all functions $f:R \rightarrow B$ such that
\begin{enumerate}
\item $f(a) = 0$ for all but finitely many $a \in R$.
\item $f(a) \wedge f(b) = 0$ for all $a \ne b$ in $R$.
\item $\bigvee {\func{Im}} f = 1$.
\end{enumerate}
Then $R[B]^\perp$ has an $R$-algebra structure given by
\begin{enumerate}
\item[(4)] $(f + g)(a) = \displaystyle{\bigvee} \{{f}(b) \wedge {g}(c): {b+c=a}\}$.
\item[(5)] ${(fg)}(a) = \displaystyle{\bigvee} \{{f}(b) \wedge {g}(c):{bc=a}\}$.
\item[(6)] ${(bf)}(a) = \bigvee\{f(c):bc=a\}$.
\end{enumerate}
As noticed by J$\acute{\mathrm{o}}$nsson in the review of \cite{Fos61}, and further elaborated by Banaschewski and Nelson \cite{BN80}, $R[B]^\perp$ is isomorphic to the Boolean power $C(X,R_{\func{disc}})$, where $X$ is the Stone space of $B$.

As the notation $(-)^\perp$ suggests, $R[B]^\perp$ encodes full orthogonal decompositions of elements of $R[B]$ into an algebra of functions from $R$ to $B$. Indeed, for a Specker $R$-algebra $S$ with a faithful generating algebra of idempotents $B$, define $(-)^\perp:S\to R[B]^\perp$ as follows. For $s\in S$, write $s=\sum_{i=1}^n a_ie_i$ in full orthogonal form, and define $s^\perp:R \rightarrow B$ by
\[
s^\perp(a) =
\left\{\begin{array}{ll}
e_i & \textrm{if } a=a_i \textrm{ for some } i, \\
0 & \textrm{otherwise.} \\
\end{array}\right.
\]
As we show in \cite{BMMO13b}, $s^\perp\in R[B]^\perp$ and $s = \sum_{a\in R} a s^\perp(a)$. Moreover, if $f \in R[B]^\perp$, then $s = \sum_{a \in R} af(a) \in S$ and $f = s^\perp $. Furthermore, the map $(-)^\perp:S\to R[B]^\perp$ is an $R$-algebra isomorphism. Thus, the interpretation of a Specker $R$-algebra $S$ in terms of $R[B]^\perp$ is convenient for applications, where the decomposition data for elements in $S$ needs to be tracked under the algebraic operations of $S$. This viewpoint plays an important role in \cite{BMMO13b}.
\end{remark}

\section{Specker algebras over an indecomposable ring}\label{sec:3}

In this section we show that when $R$ is an indecomposable ring (that is, $\Id(R)=\{0,1\}$), then the results of the previous section can be strengthened considerably. Namely, we show that for an indecomposable $R$, the category ${\bf Sp}_R$ of Specker $R$-algebras is equivalent to the category {\bf BA} of Boolean algebras, and hence, by Stone duality, is dually equivalent to the category {\bf Stone} of Stone spaces (zero-dimensional compact Hausdorff spaces). We also show that for an indecomposable $R$, Specker $R$-algebras are exactly the idempotent generated $R$-algebras which are projective as an $R$-module.

We start by pointing out that for an indecomposable $R$, the representation of Theorem~\ref{new def} of Specker $R$-algebras as Boolean powers of $R$ yields another representation of Specker $R$-algebras as idempotent generated subalgebras of $R^I$ for some set $I$.

\begin{proposition}\label{prop:2.6} \label{prop:2.4}
Let $R$ be indecomposable. A commutative $R$-algebra $S$ is a Specker $R$-algebra iff $S$ is isomorphic to an idempotent generated $R$-subalgebra of $R^I$ for some set $I$.
\end{proposition}

\begin{proof}
Let $S$ be a Specker $R$-algebra. By Theorem~\ref{new def}, $S \cong C(X,R_{\func{disc}})$ for some Stone space $X$, so $S$ is isomorphic to an idempotent generated subalgebra of $R^X$. Conversely, suppose that $S$ is an idempotent generated subalgebra of $R^I$ for some set $I$. Since $R$ is indecomposable, it is easy to see that $\Id(R^I) = \{ f \in R^I : f(i) \in \{0,1\} \ \forall i \in I\}$. From this description it is clear that each nonzero idempotent of $R^I$ is faithful. Therefore, $S$ has a faithful generating algebra of idempotents, hence is a Specker $R$-algebra.
\end{proof}

In the next lemma we characterize the idempotents of $R[B]$. For this, we view the Boolean algebra $\Id(R)$ as a Boolean ring. Then $\Id(R)[B]$ is an $\Id(R)$-algebra, which is a Boolean ring, and hence may be viewed as a Boolean algebra.

\begin{lemma}\label{lem:3.2}
Let $B$ be a Boolean algebra.
\begin{enumerate}
\item $s\in\func{Id}(R[B])$ iff $s=\sum a_i y_{e_i}$ with the $a_i \in \Id(R)$ distinct and the $e_i \in B$ a full orthogonal set.
\item $\Id(R[B]) \cong \Id(R)[B]$ as Boolean algebras.
\item $\Id(R[B])$ is isomorphic to the coproduct of $\Id(R)$ and $B$.
\item If $R$ is indecomposable, then $i_B:B\to\func{Id}(R[B])$ is a Boolean isomorphism.
\end{enumerate}
\end{lemma}

\begin{proof}
(1) Let $s = \sum a_i y_{e_i}$ with $a_i \in \Id(R)$ and the $e_i \in B$ orthogonal. Then $y_{e_i}y_{e_j} = 0$ for $i\ne j$, and so $s^2 = \sum a_i^2 (y_{e_i})^2 =\sum a_i y_{e_i} = s$. Thus, $s$ is idempotent. For the converse, let $s \in R[B]$ be idempotent. By Lemma~\ref{lem:faithful}, the $y_e$ form a faithful generating algebra of idempotents of $R[B]$, thus by Lemma~\ref{lem:2.1}, we may write $s = \sum a_i y_{e_i}$ with the $a_i \in R$ distinct and the $e_i\in B$ a full orthogonal set. If $i\ne j$, then $y_{e_i}y_{e_j} = y_{e_i \wedge e_j} = y_0 = 0$. Therefore, $s^2 = \sum a_i^2 y_{e_i}$. By Remark~\ref{uniqueness}.2, the equation $s^2 = s$ implies $a_i^2 = a_i$ for each $i$, so $a_i \in \Id(R)$.

(2) The inclusion map $\iota : \Id(R) \to \Id(R[B])$ is a Boolean homomorphism, and so is a ring homomorphism of Boolean rings. Viewing $\Id(R[B])$ as an $\Id(R)$-algebra, the Boolean homomorphism $i_B : B \to \Id(R[B])$ sending $e$ to $y_e$ extends by Lemma~\ref{UMP} to an $\Id(R)$-algebra homomorphism $\alpha : \Id(R)[B] \to \Id(R[B])$. By (1), Lemma~\ref{lem:faithful}, and Lemma~\ref{lem:2.1}, $\alpha$ is an isomorphism of Boolean rings, hence an isomorphism of Boolean algebras.

(3) Let $C$ be a Boolean algebra and suppose $\sigma : \Id(R) \to C$ and $\rho : B \to C$ are Boolean homomorphisms. Then $C$, viewed as a Boolean ring, is an $\Id(R)$-algebra, and Lemma~\ref{UMP} yields an $\Id(R)$-algebra homomorphism $\eta : \Id(R)[B] \to C$, sending $ay_e$ to $\sigma(a)\rho(e) = \sigma(a) \wedge \rho(e)$.
Clearly $\eta$ is a unique Boolean homomorphism extending $\iota$ and $i_B$. Thus, by (2), $\Id(R[B]) \cong \Id(R)[B]$ is the coproduct of $\Id(R)$ and $B$.
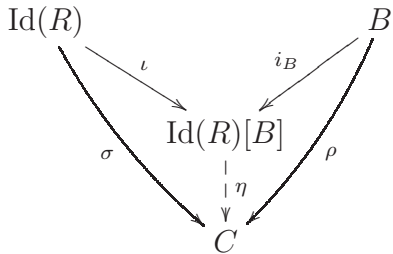
\[
\xymatrix{
\Id(R) \ar@/_/[rdd]_{\sigma} \ar[rd]^{\iota} && B \ar@/^/[ldd]^{\rho} \ar[ld]_{i_B} \\
& \Id(R)[B] \ar@{-->}[d]^{\eta} & \\
& C}
\]

(4) Since $R$ is indecomposable, it follows from (1) that $\Id(R[B])=\{y_e:e\in B\}$. Now apply Lemma~\ref{lem:faithful}.
\end{proof}

\begin{remark}
As follows from the proofs of Lemma~\ref{lem:3.2} and Theorem~\ref{new def}, for Boolean algebras $B$ and $C$, the coproduct of $B$ and $C$ in {\bf BA} can be described as the Boolean power of $B$ by $C$. In particular, we see that $B[C] \cong C[B]$. This isomorphism is contained in \cite[Exercise IV.5.1]{BS81}.
\end{remark}

%


As promised in Remark~\ref{unique B}, we next show that a faithful generating algebra of idempotents of a Specker $R$-algebra is not unique.

\begin{example}\label{example}
Suppose that $R$ is not an indecomposable ring and let $B = \{0,1,e,\lnot e\}$ be the four-element Boolean algebra. By Lemma~\ref{lem:faithful}, $\{y_b : b \in B\}$ is a faithful generating algebra of idempotents of $R[B]$. Let $a \in R$ be an idempotent with $a \ne 0,1$ and set $g = ay_e + (1-a)y_{\lnot e}$. By Lemma~\ref{lem:3.2}.1, $g$ is an idempotent in $R[B]$. Also, $1-g = (1-a)y_e + ay_{\lnot e}$. To see that $g$ is faithful, suppose that $c \in R$ with $cg = 0$. Then $c(1-a)y_e + cay_{\lnot e} = 0$. By Remark~\ref{uniqueness}.2, $c(1-a) = ca = 0$, which forces $c = 0$. A similar argument shows $1-g$ is faithful. Let $C = \{0,1,g,1-g\}$. Since $y_e = ag + (1-a)(1-g)$, we see that $C$ is a faithful generating algebra of idempotents of $R[B]$ different than $\{y_b : b \in B\}$.\end{example}

On the other hand, we next prove that a faithful generating algebra of idempotents of a Specker $R$-algebra is unique up to isomorphism.

\begin{theorem}\label{prop:2.8}
Let $S$ be a Specker $R$-algebra. If $B$ and $C$ are both faithful generating algebras of idempotents of $S$, then $B$ is isomorphic to $C$.
\end{theorem}

\begin{proof}
We identify $S$ with $R[B]$. Let $P$ be a prime ideal of $R$ and let $PS$ be the ideal of $S$ generated by $P$. Thus, $PS$ consists of the sums of elements of the form $ps$ with $p \in P$ and $s \in S$. We show that $S/PS \cong (R/P)[B]$. We note that $(R/P)[B]$ is an $R$-algebra, where scalar multiplication is given by $a\cdot \left(\sum (b_i + P)y_{e_i}\right) = \sum (ab_i+P)y_{e_i}$ for $a,b_i \in R$ and $e_i \in B$. By Lemma~\ref{UMP}, the identity homomorphism $B \to B$ lifts to an $R$-algebra homomorphism $\alpha : R[B] \to (R/P)[B]$. It is clear that $\alpha$ is onto, and $\ker(\alpha)$ contains $PS$. If $s \in \ker(\alpha)$, write $s = \sum a_i y_{e_i}$ in the unique full orthogonal form. Then $0 = \alpha(s) = \sum (a_i + P)y_{e_i}$. By Remark~\ref{uniqueness}.2, each $a_i \in P$, so $s \in PS$. Therefore, $\ker(\alpha)=PS$, and so $S/PS \cong (R/P)[B]$. Now, since $R/P$ is a domain, it is indecomposable. Thus, by Lemma~\ref{lem:3.2}.4, $B \cong \Id((R/P)[B]) \cong \Id(S/PS)$. Applying the same argument for $C$, for any prime ideal $P$ of $R$, we then get $B \cong \Id(S/PS) \cong C$, so $B \cong C$.
\end{proof}

Example~\ref{example} and Theorem~\ref{prop:2.8} show that while a faithful generating algebra $B$ of idempotents of a Specker $R$-algebra may not be unique, it is unique up to isomorphism. In the following theorem we show that if $R$ is indecomposable, then a Specker $R$-algebra $S$ has a unique faithful generating algebra of idempotents, namely $\Id(S)$.

\begin{theorem}\label{faithful}
Let $R$ be indecomposable. An idempotent generated commutative $R$-algebra $S$ is a Specker $R$-algebra iff each nonzero idempotent in $\Id(S)$ is faithful. Consequently, if $S$ is a Specker $R$-algebra, then $\Id(S)$ is the unique faithful generating algebra of idempotents of $S$.
\end{theorem}

\begin{proof}
If each nonzero idempotent of $S$ is faithful, then $\Id(S)$ is a faithful generating algebra of idempotents of $S$, and so $S$ is a Specker $R$-algebra. Conversely, suppose that $S$ is a Specker $R$-algebra. Then $S$ has a faithful generating algebra of idempotents $B$, and we identify $S$ with $R[B]$.
Because $R$ is indecomposable, Lemma~\ref{lem:3.2}.4 implies that $\Id(S) = \{ y_e : e \in B\}$. Thus, by Lemma~\ref{lem:faithful}, each nonzero idempotent of $S$ is faithful.
\end{proof}

The considerations of the previous section give rise to two functors $\mathcal I:{\bf Sp}_R\to{\bf BA}$ and $\mathcal S:\mathbf{BA} \to \mathbf{Sp}_R$. The functor $\mathcal I$ associates with each $S\in{\bf Sp}_R$ the Boolean algebra $\Id(S)$ of idempotents of $S$, and with each $R$-algebra homomorphism $\alpha:S\to S'$ the restriction $\mathcal I(\alpha)=\alpha|_{\Id(S)}$ of $\alpha$ to $\Id(S)$. The functor $\mathcal S$ associates with each $B\in \mathbf{BA}$ the Specker $R$-algebra $R[B]$, and with each Boolean homomorphism $\sigma:B\to B'$ the induced $R$-algebra homomorphism $\alpha:R[B] \to R[B']$ that sends each $y_e$ to $y_{\sigma(e)}$.

\begin{lemma}\label{adjoint}
The functor $\mathcal S$ is left adjoint to the functor $\mathcal I$.
\end{lemma}

\begin{proof}
By definition, $\mathcal{I}(\mathcal{S}(B)) = \Id(R[B])$. By \cite[Ch.~IV, Thm.~1.2]{Mac71}, the universal mapping property established in Lemma \ref{UMP} is equivalent to the fact that $\mathcal{S}$ is left adjoint to $\mathcal{I}$.
\end{proof}

We show that the functors $\mathcal I,\mathcal S$ form an equivalence of ${\bf Sp}_R$ and {\bf BA} precisely when $R$ is indecomposable.

\begin{theorem}\label{thm:3.4}
The following are equivalent.
\begin{enumerate}
\item $R$ is indecomposable.
\item $\mathcal I \circ \mathcal S \cong 1_{\bf BA}$.
\item $\mathcal S \circ \mathcal I \cong 1_{{\bf Sp}_R}$.
\item The functors $\mathcal I$ and $\mathcal S$ yield an equivalence of $\mathbf{Sp}_R$ and $\mathbf{BA}$.
\end{enumerate}
\end{theorem}

\begin{proof}
(1) $\Leftrightarrow$ (2): Suppose that $R$ is indecomposable. We have $\mathcal I(\mathcal S(B))=\func{Id}(R[B])$. By Lemma~\ref{lem:3.2}.4, $i_B:B\to\func{Id}(R[B])$ is a Boolean isomorphism, and by Lemma~\ref{adjoint}, $i_B$ is natural, so (2) follows. Conversely, if (2) holds, then $\Id(R[{\bf 2}]) \cong {\bf 2}$. Observe that $R[{\bf 2}]\cong R$. To see this, let ${\bf 2}=\{0,1\}$ and recall from Definition~\ref{def:R[B]} the ideal $I_{\bf 2}$ defining $R[{\bf 2}]$. By listing all the generators for $I_{\bf 2}$, we see that $I_{\bf 2}$ is generated by $x_0,x_1-1$. Therefore, $R[{\bf 2}] \cong R[x_0,x_1]/(x_0,x_1-1) \cong R$. Thus, ${\bf 2} \cong \Id(R[{\bf 2}]) \cong \Id(R)$, which shows that $R$ is indecomposable.

(1) $\Leftrightarrow$ (3): Suppose that $R$ is indecomposable. We have $\mathcal S(\mathcal I(S))=R[\Id(S)]$ for each Specker $R$-algebra $S$. Furthermore, by Theorem~\ref{faithful}, $\Id(S)$ is a faithful generating algebra of idempotents of $S$. Consequently, the $R$-algebra homomorphism $\alpha_S : R[\Id(S)] \to S$ sending $y_e$ to $e$ for each $e \in \Id(S)$ is an isomorphism. By Lemma~\ref{adjoint}, $\alpha_S$ is natural, so (3) follows. Conversely, if (3) holds, then $R[\Id(R)] \cong R$ via $\alpha_R$. If $e \ne 0,1$ is an idempotent in $R$, then $\alpha_R(ey_{\lnot e}) = e(\lnot e) = 0$, a contradiction to Lemma~\ref{lem:faithful}. Thus, $\Id(R) = \{0,1\}$, so $R$ is indecomposable.

(1) $\Leftrightarrow$ (4): In view of Lemma~\ref{adjoint}, (2) and (3) together are equivalent to (4). Thus, by what we have proven already, (1) implies both (2) and (3), so implies (4). Conversely, if (4) holds, then (2) holds, so (1) holds as (1) is equivalent to (2).
%
\end{proof}

\begin{corollary}\label{cor:3.5}
If $R$ is indecomposable, then ${\bf Sp}_R$ is dually equivalent to {\bf Stone}.
\end{corollary}

\begin{proof}
By Theorem~\ref{thm:3.4}, $\mathbf{Sp}_R$ is equivalent to $\mathbf{BA}$. By Stone duality, $\mathbf{BA}$ is dually equivalent to {\bf Stone}. Combining these two results yields that ${\bf Sp}_R$ is dually equivalent to {\bf Stone}.
\end{proof}

\begin{remark}
In \cite[Sec.~7]{Rib69}, Ribenboim defines the category of Boolean powers of $\mathbb{Z}$ and proves that this category is equivalent to {\bf BA}. In view of  Theorem~\ref{new def} and Remark~\ref{Foster remark}, Ribenboim's result is a particular case of Theorem~\ref{thm:3.4}. Similarly, it follows from \cite[Sec.~5]{BMO13a} that ${\bf Sp}_\mathbb{R}$ is equivalent to {\bf BA}. Again, this result is a particular case of Theorem~\ref{thm:3.4}. Moreover, by Theorem~\ref{new def}, Specker $\mathbb{R}$-algebras are isomorphic to Boolean powers of $\mathbb{R}$.
\end{remark}

As noted in the proof of Corollary~\ref{cor:3.5}, the functors ${\mathcal{I}}$ and ${\mathcal{S}}$ of Theorem~\ref{thm:3.4} compose with the functors of Stone duality to give functors between ${\bf Sp}_R$ and {\bf Stone}. The resulting contravariant functor from {\bf Stone} to ${\bf Sp}_R$ is the Boolean power functor $(-)^*:{\bf Stone}\to{\bf Sp}_R$  that associates with each $X\in{\bf Stone}$ the Boolean power $X^*=C(X,R_{\func{disc}})$, and with each continuous map $\varphi:X\to Y$ the $R$-algebra homomorphism $\varphi^*:Y^*\to X^*$ given by $\varphi(f)=f\circ\varphi$. The functor $(-)_*:{\bf Sp}_{R} \rightarrow {\bf Stone}$ sends the Specker $R$-algebra $S$ to the Stone space of $\Id(S)$ and associates with each $R$-algebra homomorphism $S \rightarrow T$, the continuous map from the Stone space of $\Id(S)$ to the Stone space of $\Id(T)$. By Corollary~\ref{cor:3.5}, these two functors yield a dual equivalence when $R$ is indecomposable. In general, we have the following diagram.

\[
\xymatrix{
{\bf BA}  \ar@/_/[ddrr]_{\mathcal S} \ar@{<->}[rrrr]^{\textrm{Stone Duality}}   &&&& {\bf Stone} \ar@/^/[ddll]^{(-)^*}  \\
&&&& \\
&& {\bf Sp}_R \ar@/^/[rruu]^{(-)_*} \ar@/_/[uull]_{\mathcal I}&& }
\]

We show in Proposition~\ref{new top} that the functor $(-)_*:{\bf Sp}_R\to{\bf Stone}$ has a natural interpretation when $R$ is indecomposable, one that does not require reference to $\Id(S)$. Let $S$ be a Specker $R$-algebra and let $\Hom_R(S,R)$ be the set of $R$-algebra homomorphisms from $S$ to $R$. We define a topology on $\Hom_R(S,R)$ by declaring $\{U_s:s\in S\}$ as a subbasis, where $U_s = \{\alpha \in \Hom_R(S,R):\alpha(s)=0\}$. We also recall that the Stone space of a Boolean algebra $B$ can be described as the set $\Hom(B,{\bf 2})$ of Boolean homomorphisms from $B$ to {\bf 2}, topologized by the basis $\{Z(e):e\in B\}$, where $Z(e)=\{\sigma\in\Hom(B,{\bf 2}):\sigma(e)=0\}$.


\begin{proposition}\label{new top}
Let $R$ be indecomposable, and let $S$ be a Specker $R$-algebra. Then $\Hom_R(S,R)$ is homeomorphic to $\Hom(\Id(S),{\bf 2})$.
\end{proposition}

\begin{proof}
Set $B = \func{Id}(S)$ and define $\varphi : \Hom_R(S,R) \to \Hom(B,{\bf 2})$ by $\varphi(\alpha) = \alpha|_B$. By Theorem~\ref{new def}, $\varphi$ is onto. It is 1-1 because if $\alpha|_B = \beta|_B$, then $\alpha, \beta$ are $R$-algebra homomorphisms which agree on a generating set of $S$, so $\alpha = \beta$.
We have
\[
\varphi^{-1}(Z(e)) = \{ \alpha \in \Hom_R(S,R) : \alpha(e) = 0\} = U_e,
\]
which proves that $\varphi$ is continuous. It also shows that $\varphi(U_e) = Z(e)$. Now, let $s \in S$. If $s = 0$, then $U_s = \Hom_R(S,R)$, so $\varphi(U_s) = \Hom(B,{\bf 2})$ is open. Otherwise, we may write $s = \sum_i a_i e_i$ with the $a_i\in R$ nonzero and the $e_i\in B$ orthogonal. If $\alpha \in U_s$, then $s \in \ker(\alpha)$, so $a_ie_i = se_i \in \ker(\alpha)$. Thus, $e_i \in \ker(\alpha)$ since otherwise $\alpha(e_i) = 1$, and this contradicts $a_i \ne 0$. Therefore, $\alpha \in U_{e_1}\cap \dots \cap U_{e_n}$. The reverse inclusion is obvious. Thus, $U_s = U_{e_1}\cap\dots\cap U_{e_n}$, and so $\varphi(U_s) = Z(e_1) \cap \cdots \cap Z(e_n)$. Since the $U_s$ form a subbasis for $\Hom_R(S,R)$, this proves that $\varphi^{-1}$ is continuous. Consequently, $\varphi$ is a homeomorphism.
\end{proof}

It follows that when $R$ is indecomposable, the space $\Hom_R(S,R)$ of a Specker $R$-algebra $S$ is homeomorphic to the Stone space of $\Id(S)$. This allows us to describe the contravariant functor $(-)_*:{\bf Sp}_R \rightarrow {\bf Stone}$ as follows. Associate with each $S \in {\bf Sp}_R$ the Stone space $S_* = \Hom_R(S,R)$, and with each $R$-algebra homomorphism $\alpha:S\to T$, the continuous map $\alpha_*:T_*\to S_*$ given by $\alpha_*(\delta)=\delta\circ\alpha$ for each $\delta\in T_*=\Hom_R(T,R)$. Thus, we have a description of $(-)_*$ that does not require passing to idempotents.

We conclude this section by giving a module-theoretic characterization of Specker $R$-algebras for indecomposable $R$. Bergman \cite[Cor.~3.5]{Ber72} has shown that a Boolean power $C(X,R_{\mathrm{disc}})$ of the ring $R$ is a free $R$-module having a basis of idempotents. Thus, by Theorem~\ref{new def}, every Specker $R$-algebra is a free $R$-module having a basis of idempotents. We prove in the next theorem that the converse of the corollary is true when $R$ is indecomposable, and that in this case freeness is equivalent to projectivity.

\begin{theorem}\label{thm:3.11}
Let $R$ be indecomposable and let $S$ be an idempotent generated commutative $R$-algebra. Then the following are equivalent.
\begin{enumerate}
\item $S$ is a Specker $R$-algebra.
\item $S$ is a free $R$-module.
\item $S$ is a projective $R$-module.
\end{enumerate}
\end{theorem}

\begin{proof}
As was discussed above, (1) $\Rightarrow$ (2) follows from \cite[Cor.~3.5]{Ber72} and Theorem~\ref{new def}, and (2) $\Rightarrow$ (3) is obvious. It remains to show that (3) $\Rightarrow$ (1). Let $B = \Id(S)$. By Lemma~\ref{UMP}, the inclusion $B \rightarrow S$ lifts to an $R$-algebra homomorphism $\alpha:R[B] \rightarrow S$. Since $S$ is generated by $B$, we have that $\alpha$ is onto. In particular, for each idempotent $e \in S$, we have $\alpha(y_e) = e$. Now since $S$ is a projective $R$-module, there exists an $R$-module homomorphism $\beta:S \rightarrow R[B]$ such that $\alpha(\beta(s)) = s$ for all $s \in S$. Let $e$ be an idempotent in $S$. Write $\beta(e) = \sum a_iy_{e_i}$ with the $a_i \in R$ and the $e_i\in\Id(S)$ orthogonal. Then $e = \alpha(\beta(e)) = \sum a_i\alpha(y_{e_i}) = \sum a_ie_i$.

First observe that for every $a \in \func{ann}_R(e)$, we have $aa_1 = \cdots = aa_n  =0$. Indeed, for $a \in \func{ann}_R(e)$, we have $0 =\beta(ae) = a\beta(e) = \sum aa_iy_{e_i}$, so that since by Lemma~\ref{lem:faithful}, each $y_{e_i}$ is faithful, we have $aa_i = 0$. This in turn implies that if $\func{ann}_R(e) \ne 0$, then $(a_1,\ldots,a_n)R$ is a proper ideal of $R$, as every element in $\func{ann}_R(e)$ annihilates $(a_1,\ldots,a_n)R$. We use these observations to show that either $\func{ann}_R(e) = 0$ or $e = 0$.

Suppose $\func{ann}_R(e) \ne 0$. We show that $e = 0$. First we claim that $\func{ann}_R(e) + (a_1,\ldots,a_n)R = R$. Let $M$ be a maximal ideal of $R$ containing $a_1,\ldots,a_n$. Since $e = \sum a_ie_i$ is an orthogonal decomposition of $e$ and $e$ is idempotent, it follows that $a_ie_i = a_i^2e_i$, and hence $a_i(1-a_i)e_i = 0$ for each $i$. As each $a_i \in M$, the image of $1-a_i$ in the localization $R_M$ is a unit, so $a_i(1-a_i)e_i = 0$ implies that the image of $a_ie_i$ in the ring $S_M$ is $0$. Since this holds for each $i$, it must be that the image of $e = \sum a_ie_i$ in $S_M$ is $0$. But then there exists $b \in R \setminus M$ such that $be = 0$; i.e., $\func{ann}_R(e) \not \subseteq M$. This proves that no maximal ideal of $R$ containing $(a_1,\ldots,a_n)R$ also contains ${\func{ann}}_R(e)$. Hence, ${\func{ann}}_R(e) + (a_1,\ldots,a_n)R = R$, so that there exist $a \in \func{ann}_R(e)$ and $b_1,\ldots,b_n \in R$ such that $a + \sum a_ib_i = 1$. By assumption ${\func{ann}}_R(e) \ne 0$, so  as established above,    $(a_1,\ldots,a_n)R$ is a proper ideal of $R$. In particular,  $1 \ne \sum a_ib_i$, so since $1=a + \sum a_ib_i$, this forces $a \ne 0$. As noted above, $aa_1 = \cdots = aa_n = 0$. Thus, $a(\sum a_ib_i) = 0$, so that multiplying both sides of the equation $a + \sum a_ib_i = 1$ by $a$, yields $a^2 = a$. Therefore, $a\in\Id(R)$, and since $R$ is indecomposable and $a\ne 0$, this forces $a = 1$. But $ae = 0$, so we conclude that $e =0$. This proves that every nonzero idempotent in $S$ is faithful, and hence, as $S$ is idempotent generated, $S$ is a Specker $R$-algebra.
\end{proof}

\begin{remark}
The assumption of indecomposability in the theorem is necessary: If $R$ is not indecomposable, then there exists an idempotent $a$ in $R$ distinct from $0,1$, so that $R = aR \oplus (1-a)R$, and hence $S := R/aR$ is a projective $R$-module that is generated as an $R$-algebra by the idempotent $1+aR$. Yet $1+aR$ is not faithful because it is annihilated by $a$, so $S$ is not a Specker $R$-algebra.
\end{remark}

\section{Specker algebras over a domain}

As follows from the previous section, having $R$ indecomposable allows one to prove several strong results about Specker $R$-algebras.
Some of these results can be strengthened further provided $R$ is a domain. In this section we consider in more detail the case when $R$ is a domain. We first show that among idempotent generated commutative algebras over a domain $R$, the Specker $R$-algebras are simply those that are torsion-free $R$-modules. We give
then a necessary and sufficient condition for a Specker $R$-algebra $S$ to be a weak Baer ring and a Baer ring. For a domain $R$, the characterization of Specker $R$-algebras that are Baer rings yields a characterization of injective objects as well as the construction of injective hulls in ${\bf Sp}_R$. In addition, it provides a description of $S_*=\mathrm{Hom}_R(S,R)$ by means of minimal prime ideals of $S$.

\begin{proposition}\label{tfree char}
Let $R$ be a domain and let $S$ be an idempotent generated commutative $R$-algebra. Then $S$ is a Specker $R$-algebra iff $S$ is a torsion-free $R$-module.
\end{proposition}

\begin{proof}
As discussed before Theorem~\ref{thm:3.11}, a Specker $R$-algebra $S$ is a free $R$-module, and hence with $R$ a domain, $S$ is torsion-free. Conversely, if $S$ is an idempotent generated commutative $R$-algebra that is torsion-free, then nonzero idempotents are faithful, and hence by Theorem~\ref{faithful}, $S$ is a Specker $R$-algebra.
\end{proof}


Next we recall the well-known definition of a Baer ring and a weak Baer ring in the case of a commutative ring.

\begin{definition}\label{Baer def}
A commutative ring $R$ is a \emph{Baer ring} if the annihilator ideal of each subset of $R$ is a principal ideal generated by an idempotent, and $R$ is a \emph{weak Baer ring} if the annihilator ideal of each element of $R$ is a principal ideal generated by an idempotent.
\end{definition}

As we noted after Definition~\ref{def of Specker}, we will view $R$ as an $R$-subalgebra of each Specker $R$-algebra $S$.

\begin{theorem}\label{char-of-Baer}
Let $S$ be a Specker $R$-algebra.
\begin{enumerate}
\item $S$ is weak Baer iff $R$ is weak Baer.
\item $S$ is Baer iff $S$ is weak Baer and $\Id(S)$ is a complete Boolean algebra.\end{enumerate}
\end{theorem}

\begin{proof}
(1) Let $B$ be a faithful generating algebra of idempotents for $S$. Suppose that $S$ is weak Baer and let $a \in R$. Then there is $e \in \Id(S)$ with $\ann_S(a) = eS$. By Lemma~\ref{lem:3.2}.1 and Theorem~\ref{new def}, we may write $e = \sum b_ie_i$ in full orthogonal form with $b_i \in \Id(R)$ and the $e_i \in B$. Since $0 = ae = \sum (ab_i)e_i$, by Remark~\ref{uniqueness}.2 we see that $ab_i = 0$ for all $i$. Therefore, $b_j \in eS$ for each $j$, hence $b_j = es$ for some $s \in S$. Then $eb_j = e(es) = es = b_j$. The equation $eb_j = b_j$ yields $\sum (b_i b_j)e_i = b_j = \sum b_j e_i$ since $\sum e_i = \bigvee e_i = 1$. Remark~\ref{uniqueness}.2 yields $b_ib_j = b_j$. Applying the same argument to $eb_i = b_i$ gives $b_i b_j = b_i$, so $b_i = b_j$ for each $i,j$. Thus, $e = \sum b_i e_i = b_1 \sum e_i = b_1$. Consequently, $e = b_1 \in R$. From this it follows that $\ann_R(a) = b_1R$ is generated by the idempotent $b_1$, so $R$ is weak Baer.

Conversely, suppose that $R$ is weak Baer. Let $s \in S$ and write $s = \sum a_i e_i$ in full orthogonal form with the $a_i \in R$ distinct and the $e_i\in B$. Since $R$ is weak Baer, $\ann_R(a_i) = b_iR$ for some idempotent $b_i \in R$.  Let $e = \sum b_i e_i$. By Lemma~\ref{lem:3.2}.1, $e$ is an idempotent in $S$. We claim that $\ann_S(s) = eS$. We have $es = \left(\sum b_ie_i\right)\left(\sum a_i e_i\right) = \sum (b_ia_i)e_i = 0$ because the $e_i$ are orthogonal and the $b_i$ annihilate the $a_i$. So $eS \subseteq \ann_S(s)$. To prove the reverse inclusion, we first show that if $b \in R$ and $g \in \Id(S)$ with $bg \in \ann_S(s)$, then $bg \in eS$. If $bgs = 0$, then $\sum (ba_i)(e_ig) = 0$. Thus, by Remark~\ref{uniqueness}.2, for each $i$ with $e_ig \ne 0$ we have $ba_i = 0$. When this occurs, $b \in b_iR$, so $b = bb_i$. Consequently,
\begin{eqnarray*}
e(bg) &=&\left(\sum b_i e_i\right)bg = \sum (b_ib)(e_ig) = \sum b(e_ig) = \left(\sum e_i \right)bg \\
&=& 1\cdot bg  = bg.
\end{eqnarray*}
Thus, $bg \in eS$. In general, if $t \in \ann_S(s)$, write $t = \sum c_j f_j$ in orthogonal form. Then each $c_j f_j = tf_j \in \ann_S(s)$. By the previous argument, each $c_j f_j \in eS$, so $t \in eS$. This proves that $\ann_S(s) = eS$, so $S$ is weak Baer.

(2) First suppose that $S$ is weak Baer and $\Id(S)$ is complete, and let $I \subseteq S$. Then $\func{ann}_S(I) = \bigcap_{s \in I} \func{ann}_S(s)$. Since $S$ is weak Baer, there is $e_s \in \func{Id}(S)$ with $\func{ann}_S(s) = e_sS$. Consequently, $\func{ann}_S(I) = \bigcap_{s \in I} e_sS$. Let $e = \bigwedge e_s$. We show that $\func{ann}_S(I)= eS$. Since $e\le e_s$ for each $s$, we have $ee_s = e$, so $e \in \bigcap e_sS = \func{ann}_S(I)$. Conversely, let $t \in \func{ann}_S(I)$. Then $ts = 0$ for all $s \in I$, so $t \in e_sS$ for each $s$, which yields $te_s = t$. Let $t = \sum b_i f_i$ be the full orthogonal decomposition of $t$ with the $b_i\in R$ distinct and the $f_i \in B$. Then $te_s = t$ yields $\sum b_i f_ie_s = \sum b_if_i$. By Remark~\ref{uniqueness}.2, $f_i e_s = f_i$, so $f_i \le e_s$ for each $s$. Therefore, $f_i \le e$, so $f_ie =f_i$. Since this is true for all $i$, we have $te = t$. This yields $t \in eS$. Thus, $\func{ann}_S(I) = eS$, and so $S$ is Baer.

Next suppose that $S$ is Baer. Then $S$ is weak Baer. Let $\{e_i : i \in I\}$ be a family of idempotents of $S$. Set $K = \{1-e_i : i \in I\}$. Then $\func{ann}_S(1-e_i) = e_iS$, so $\func{ann}_S(K) = \bigcap e_iS$. Since $S$ is Baer, $\func{ann}_S(K) = eS$ for some $e \in \func{Id}(S)$. We show that $e=\bigwedge e_i$. First, as $e \in \func{ann}_S(K)$, we have $ee_i = e$, so $e \le e_i$. Thus, $e$ is a lower bound of the $e_i$. Next, let $f \in \func{Id}(S)$ be a lower bound of the $e_i$. Then $fe_i = f$, so $(1-e_i)f = 0$. Therefore, $f \in \func{ann}_S(K) = eS$. This implies that $fe = f$, so $f \le e$. Thus, $e = \bigwedge_i e_i$. Consequently, $\Id(S)$ is a complete Boolean algebra.
\end{proof}

\begin{corollary}\label{cor:4.4}
Let $S$ be a Specker $R$-algebra.
\begin{enumerate}
\item If $R$ is indecomposable, then $S$ is Baer iff $R$ is a domain and $\Id(S)$ is a complete Boolean algebra.
\item If $R$ is a domain, then $S$ is a weak Baer ring.
\end{enumerate}
\end{corollary}

\begin{proof}
(1) By Theorem~\ref{char-of-Baer}.2, $S$ is Baer iff $S$ is weak Baer and $\Id(S)$ is complete. By Theorem~\ref{char-of-Baer}.1, $S$ is weak Baer iff $R$ is weak Baer. Now, since $R$ is indecomposable, the only idempotents are $0,1$, so if $R$ is weak Baer, then $\func{ann}_R(a) = 0$ for each $a \in R$, which means each nonzero element is a non-zero divisor, so $R$ is a domain. Conversely, if $R$ is a domain, then trivially $R$ is Baer. Thus, (1) follows.

(2) This follows from Theorem~\ref{char-of-Baer}, since a domain is a Baer ring.
\end{proof}

Next we show that when $R$ is a domain, then $S_*$ is also homeomorphic to the space $\Min(S)$ of minimal prime ideals of $S$ with the subspace  topology inherited from the Zariski topology on the prime spectrum of $S$. Therefore, the closed sets of $\Min(S)$ are the sets of the form $Z(I)=\{P\in\Min(S):I\subseteq P\}$ for some ideal $I$ of $S$.

\begin{lemma}\label{min primes}
When $R$ is a domain, the following are equivalent for a prime ideal $P$ of a Specker $R$-algebra $S$.
\begin{enumerate}
\item $P$ is a minimal prime ideal of $S$.
\item $P \cap R = 0$.
\item Every element of $P$ is a zero divisor.
\end{enumerate}
\end{lemma}

\begin{proof}
(1) $\Rightarrow$ (2): Suppose $P$ is a minimal prime ideal of $S$. Then every element of $P$ is a zero divisor in $S$ (see, e.g., \cite[Cor.~1.2]{HJ65}). Thus, if $a \in P \cap R$, then there exists $0 \ne s \in S$ such that $as = 0$. But by Proposition~\ref{tfree char}, $S$ is a torsion-free $R$-module, so necessarily $a =0$, and hence $P \cap R = 0$.

(2) $\Rightarrow$ (3): Let $s \in P$ be nonzero and write $s = \sum a_ie_i$ in orthogonal form with the $a_i$ distinct and nonzero. Then $a_ie_i = se_i \in P$ for each $i$, so since $P \cap R = 0$, it must be that $e_i \in P$. Thus, $1 \ne \sum e_i$, and hence since $(1-\sum e_i)s=0$, we see that $s$ is a zero divisor in $S$.

(3) $\Rightarrow$ (1): This is a general fact about weak Baer rings; see \cite[Lem.\ 3.8]{HJ65}.
\end{proof}

\begin{theorem}\label{min remark}
If $R$ is a domain and $S$ is a Specker $R$-algebra, then $S_*$ is homeomorphic to $\Min(S)$.
\end{theorem}

\begin{proof}
By Proposition~\ref{new top}, we identify $S_*$ with $\Hom_R(S,R)$. If $\alpha \in S_*$, then as $R$ is a domain, $P:=\ker(\alpha)$ is a prime ideal. Moreover, $P\cap R=0$ since $a\in R$ implies $\alpha(a)=a$.
Consequently, by Lemma~\ref{min primes}, $P$ is a minimal prime ideal of $S$. Conversely, if $P$ is a minimal prime ideal of $S$, then consider the canonical $R$-algebra homomorphism $R \to S \to S/P$. By Lemma~\ref{min primes}, $R \cap P = 0$, so this homomorphism is 1-1.
To see that it is onto observe that since $S/P$ is a domain, $e+P = 0+P,1+P$ for each idempotent $e\in S$. Therefore, $S/P$ is generated over $R$ by $1+P$, and so the homomorphism $R \to S/P$ is onto. Thus, there is $\alpha \in S_*$ with $P = \ker(\alpha)$. This shows that there is a bijection $\varphi : S_* \to \Min(S)$, given by $\varphi(\alpha) = \ker(\alpha)$. To see that $\varphi$ is continuous, if $I$ is an ideal of $S$, then
\[
\varphi^{-1}(Z(I)) = \varphi^{-1}\left( \bigcap_{s \in I} Z(s)\right) = \bigcap_{s \in I} \varphi^{-1}(Z(s)) = \bigcap_{s \in I} U_s,
\]
where the last equality follows from the proof of Proposition~\ref{new top}. It also follows from the proof of Proposition~\ref{new top} that if $s = \sum a_i e_i$ is in orthogonal form, then $U_s = U_{e_1}\cap\dots\cap U_{e_n}$. Because $U_e  = S_* - U_{\lnot e}$ for each $e \in \Id(S)$, we see that each $U_e$ is clopen, and so $U_s$ is clopen. Therefore, the equation above shows that $\varphi$ is continuous. In addition, because $\varphi$ is onto, we have
\begin{eqnarray*}
\varphi(U_s) &=& \varphi\left(\{ \alpha\in S_* : \alpha(s) = 0\}\right) = \varphi\left(\{ \alpha\in S_*: s \in \ker(\alpha)\}\right) \\
&=& \{ P \in \Min(S) : s \in P \} = Z(s).
\end{eqnarray*}
Thus, $\varphi^{-1}$ is continuous, so $\varphi$ is a homeomorphism.
\end{proof}

The equality $Z(s) = \varphi(U_s)$ in the proof above shows that $Z(s)$ is clopen in $\Min(S)$ for each $s \in S$. This contrasts the case of the prime spectrum of $S$, where $Z(s)$ is clopen iff $s$ is an idempotent.

Let $\mathbf{BSp}_R$ be the full subcategory of $\mathbf{Sp}_R$ consisting of Baer Specker $R$-algebras, let $\mathbf{cBA}$ be the full subcategory of {\bf BA} consisting of complete Boolean algebras, and let {\bf ED} be the full subcategory of {\bf Stone} consisting of extremally disconnected spaces.

\begin{theorem}\label{thm:2.12}
\begin{enumerate}
\item[]
\item When $R$ is a domain, the categories $\mathbf{BSp}_R$ and $\mathbf{cBA}$ are equivalent.
\item When $R$ is a domain, the categories ${\bf BSp}_R$ and {\bf ED} are dually equivalent.
\end{enumerate}
\end{theorem}

\begin{proof}
(1) By Corollary~\ref{cor:4.4}, when $R$ is a domain, a Specker $R$-algebra is a Baer ring iff $\Id(S)$ is a complete Boolean algebra. Now apply Theorem~\ref{thm:3.4} to obtain that the restrictions of the functors $\mathcal I$ and $\mathcal S$ yield an equivalence of $\mathbf{BSp}_R$ and $\mathbf{cBA}$.

(2)
Stone duality yields that the restriction of $(-)_*$ to ${\bf BSp}_R$ lands in {\bf ED}. When $R$ is a domain, $\Id(X^*)$ consists of the characteristic functions of clopen subsets of $X$. Stone duality and Corollary~\ref{cor:4.4} then yield that the restriction of $(-)^*$ to {\bf ED} lands in ${\bf BSp}_R$. Now apply Corollary~\ref{cor:3.5} to conclude that the restrictions of $(-)_*$ and $(-)^*$ yield a dual equivalence of ${\bf BSp}_R$ and {\bf ED}.
\end{proof}

%

Since injectives in {\bf BA} are exactly the complete Boolean algebras, as an immediate consequence of Theorem~\ref{thm:2.12}, we obtain:

\begin{corollary}\label{cor:2.13}
When $R$ is a domain, the injective objects in $\mathbf{Sp}_R$ are the Baer Specker $R$-algebras.
\end{corollary}

\begin{remark}
In fact, when $R$ is a domain, each $S\in{\bf Sp}_R$ has the injective hull in ${\bf Sp}_R$, which can be constructed as follows. Let $\func{DM}(\Id(S))$ be the Dedekind-MacNeille completion of the Boolean algebra $\Id(S)$. Then, by \cite[Prop.~3]{BB67} and Theorem~\ref{thm:3.4}, $R[\func{DM}(\Id(S))]$ is the injective hull of $S$ in ${\bf Sp}_R$.
\end{remark}

\section{Specker algebras over a totally ordered ring}\label{sec:4}

Recall (see, e.g., \cite[Ch.~XVII]{Bir79}) that a ring $R$ with a partial order $\le$ is an \emph{$\ell$-ring} (lattice-ordered ring) if (i)~$(R,\le)$ is a lattice, (ii)~$a \le b$ implies $a + c \le b + c$ for each $c$, and (iii)~$0 \le a, b$ implies $0 \le ab$. An $\ell$-ring $R$ is \emph{totally ordered} if the order on $R$ is a total order, and it is an \emph{$f$-ring} if it is a subdirect product of totally ordered rings. It is well known (see, e.g., \cite[Ch.~XVII, corollary to Thm.~8]{Bir79}) that an $\ell$-ring $R$ is an $f$-ring iff for each $a, b, c \in R$ with $a \wedge b = 0$ and $c \ge 0$, we have $ac \wedge b = 0$.

In this final section we consider the case when $R$ is a totally ordered ring. Our motivation for considering Specker algebras over totally ordered rings stems from the case when $R = {\mathbb{Z}}$ as treated by Ribenboim \cite{Rib69} and Conrad \cite{Con74}, and the case $R = {\mathbb{R}}$ studied in \cite{BMO13a}. These approaches all have in common a lifting of the order from the totally ordered ring to what is a fortiori a Specker $R$-algebra, and in all three cases the lift produces the same order. We show in  Theorem~\ref{order prop} that when $R$ is totally ordered, then there is a unique partial order on a Specker $R$-algebra that makes it into an $f$-algebra over $R$.

We start by noting that each totally ordered ring $R$ is indecomposable. To see this, we first note that if $a \in R$, then $a^2 \ge 0$, since if $a \ge 0$, then $a^2 \ge 0$, and if $a \le 0$, then $-a \ge 0$, so $a^2 = (-a)^2 \ge 0$. Now, let $e \in R$ be idempotent. Then $0 \le e$ since $e = e^2$. Now, either $e \le 1-e$ or vice-versa. If $e \le 1-e$, then multiplying by $e$ yields $e^2 \le 0$, which forces $e=0$. On the other hand, if $1-e \le e$, then multiplying by $1-e$, which is nonnegative since it is an idempotent, we get $1-e \le 0$. Like before this forces $1-e = 0$, so $e = 1$. Thus, $\Id(R) = \{0,1\}$.

Let $(S,\le)$ be a partially ordered $R$-algebra. We call $S$ an \emph{$\ell$-algebra over $R$} if $S$ is both an $\ell$-ring and an $R$-algebra such that whenever $0\le s \in S$ and $0 \le a \in R$, then $as \ge 0$. Furthermore, we call $S$ an \emph{$f$-algebra over $R$} if $S$ is both an $\ell$-algebra over $R$ and an $f$-ring.

\begin{theorem}\label{order prop}
Let $R$ be totally ordered and let $S$ be a Specker $R$-algebra. Then there is a unique partial order on $S$ for which $(S,\le)$ is an $f$-algebra over $R$.
\end{theorem}

\begin{proof}
By Theorem \ref{new def}, we identify $S$ with $C(X,R_{\func{disc}})$ for some Stone space $X$. Since $R$ is totally ordered, there is a partial order on $S$, defined by $f \le g$ if $f(x) \le g(x)$ for each $x \in X$. It is elementary to see that $S$ with this partial order is an $\ell$-algebra over $R$. Let $f, g \in S$ with $f \wedge g = 0$ and let $h \ge 0$. Then, for each $x \in X$, either $f(x) = 0$ or $g(x) = 0$. Therefore, $fh(x) = 0$ or $g(x) = 0$, so $fh \wedge g = 0$. Thus, $S$ is an $f$-ring, and so is an $f$-algebra.

To prove uniqueness, suppose we have a partial order $\le^\prime$ on $S$ for which $(S,\le^\prime)$ is an $f$-algebra over $R$. As squares in $S$ are positive \cite[Sec.~XVII, Lem.~2]{Bir79}, idempotents in $S$ are positive. Let $f \in S$ be nonzero, and write $f = \sum a_i \chi_{U_i}$ for some nonzero $a_i \in R$ and $U_i$ nonempty pairwise disjoint clopen subsets of $X$. Since the $a_i$ are distinct nonzero values of $f$, we see that $0 \le f$ iff each $a_i \ge 0$. Therefore, if $0 \le f$, then $0 \le^\prime f$. Conversely, let $0 \le^\prime f$ and let $f = \sum a_i \chi_{U_i}$ as above. Note that $f\chi_{U_j} = a_j \chi_{U_j}$ for each $j$. Since $0 \le^\prime f,\chi_{U_j}$, we have $0\le^\prime f\chi_{U_j}$, so $0\le^\prime a_j\chi_{U_j}$. As $0 \ne \chi_{U_j}$, if $a_j < 0$, then $-a_j > 0$, and since $S$ is an $f$-algebra, $0 \le^\prime (-a_j)\chi_{U_j}$. Therefore, $a_j \chi_{U_j} \le^\prime 0$. This implies $a_j \chi_{U_j} = 0$, which is impossible since the $\chi_{U_j}$ are faithful idempotents and $a_j \ne 0$. Thus, $a_j \ge 0$ for each $j$. Consequently, $0 \le f$, and so $\le^\prime$ is equal to $\le$.
\end{proof}

%

\begin{remark}
Ribenboim \cite[Thm.~5]{Rib69} shows that when $B$ is a Boolean algebra, the order on ${\mathbb{Z}}$ lifts to the Boolean power of ${\mathbb{Z}}$ by $B$ in such a way that the resulting Abelian group is an $\ell$-group. His approach is through Foster's version of Boolean powers (see Remark~\ref{Foster remark}), while Theorem~\ref{order prop} recovers his result via J$\acute{\mathrm{o}}$nsson's interpretation of Boolean powers. In this sense our proof is similar in spirit to Conrad's point of view of Specker $\ell$-groups, which emphasizes the fact that such an $\ell$-group can be viewed as a subdirect product of copies of ${\mathbb{Z}}$, and hence inherits the
order from this product; see \cite[Sec.~4]{Con74}.
\end{remark}

Let $R$ be totally ordered and let $S$ and $T$ be $\ell$-algebras over $R$. We recall that an \emph{$\ell$-algebra homomorphism} $\alpha:S\to T$ is an $R$-algebra homomorphism that is in addition a lattice homomorphism. The following corollary allows us to conclude that when $R$ is totally ordered, then an $R$-algebra homomorphism between Specker $R$-algebras is automatically an $\ell$-algebra homomorphism, thus the category of Specker $R$-algebras and $\ell$-algebra homomorphisms is a full subcategory of the category of commutative $R$-algebras and $R$-algebra homomorphisms. The corollary is motivated by a similar result for rings of real-valued continuous functions \cite[Thm.~1.6]{GJ60}, and its proof is a modification of the proof of that result.

\begin{corollary}\label{order cor}
If $S,T \in {\bf Sp}_R$, then each $R$-algebra homomorphism $\alpha:S \rightarrow T$ is an $\ell$-algebra homomorphism.
\end{corollary}

\begin{proof}
Identifying $S$ with $C(X,R_{\func{disc}})$ and using $f \ge 0$ iff $f(x) \ge 0$ for all $x \in X$, we see that the unique partial order on $S$ has positive cone
\[
\left\{ \sum a_i e_i : a_i \ge 0, e_i \in \Id(S) \right\}.
\]
From the description of the positive cone it follows that $\alpha$ is order-preserving. Let $s \in S$. We recall that the $\ell$-ring $S$ has an absolute value. Since $S = C(X,R_{\func{disc}})$, we may define it explicitly as $|s|(x) = |s(x)|$ for each $x \in X$. Then $|s|^2 = s^2$ and $\alpha(|s|)^2 = \alpha(|s|^2) = \alpha(s^2) = \alpha(s)^2$. Therefore, as $\alpha(|s|) \ge 0$ and an element of an $\ell$-ring has at most one positive square root, $\alpha(|s|) = |\alpha(s)|$ (see, e.g., \cite[Ch.\ XVII]{Bir79}). Because of the $\ell$-ring formula
\[
2(a \vee b) = a+b-|a-b|,
\]
we have $\alpha(2(a \vee b)) = \alpha(a) + \alpha(b) - |\alpha(a)-\alpha(b)|$. Consequently, $2\alpha(a \vee b) = 2(\alpha(a) \vee \alpha(b))$. Since each nonzero element of an $\ell$-group has infinite order (see, e.g., \cite[Sec.~XIII, Cor.~3.1]{Bir79}), $\alpha(a \vee b) = \alpha(a) \vee \alpha(b)$. Thus, $\alpha$ is an $\ell$-algebra homomorphism.
%
\end{proof}

We conclude this article with a few comments on the interpretation of the Boolean power representation of a Specker $R$-algebra when $R$ is a totally ordered ring. As follows from Theorem~\ref{new def}, Specker $R$-algebras are represented as $X^*= C(X,R_{\func{disc}})$, where $X$ is a Stone space and $R_{\func{disc}}$ is viewed as a discrete space. Since $R$ is totally ordered, we can equip $R$ with the interval topology. Sometimes this interval topology is discrete, e.g.~when $R=\mathbb Z$, but often it is not, e.g.~when $R=\mathbb R$. In this situation, there is another natural object to study, namely the algebra $C(X,R)$ of continuous functions from a Stone space $X$ to $R$, where $R$ has the interval topology. As the discrete topology is finer than the interval topology, we have that $C(X,R_{\func{disc}})$ is an $R$-subalgebra of $C(X,R)$. Often $C(X,R_{\func{disc}})$ is a proper $R$-subalgebra of $C(X,R)$. For example, if $X$ is the one-point compactification of $\mathbb N$, then $f:X\to\mathbb R$ given by $f(n)=1/n$ and $f(\infty)=0$ is in $C(X,\mathbb R)-C(X,\mathbb R_{\func{disc}})$. Let $FC(X,R)$ be the set of finitely-valued continuous functions from $X$ to $R$. It is obvious that $FC(X,R)$ is an $R$-subalgebra of $C(X,R)$.

\begin{proposition}\label{prop:3.4}
$C(X,R_{\func{disc}}) = FC(X,R)$.
\end{proposition}

\begin{proof}
Let $f\in C(X,R_{\func{disc}})$. Since $C(X,R_{\func{disc}})$ is a Specker $R$-algebra, $f$ is finitely-valued. Let $\{a_1,\dots,a_n\}$ be the values of $f$. Then $\{f^{-1}(a_1),\dots,f^{-1}(a_n)\}$ is a partition of $X$. As the interval topology is Hausdorff, each $f^{-1}(a_i)$ is closed, so $\{f^{-1}(a_1),\dots,f^{-1}(a_n)\}$ is a partition of $X$ into finitely many closed sets. This implies that each $f^{-1}(a_i)$ is clopen. Therefore, $f\in FC(X,R)$. Conversely, if $f\in FC(X,R)$, then $f$ is a finitely-valued function in $C(X,R)$. Using again that the interval topology is Hausdorff, we conclude that $f\in C(X,R_{\func{disc}})$. Thus, $FC(X,R)=C(X,R_{\func{disc}})$.
\end{proof}

\begin{remark}
One way to think about $FC(X,R)$ is as piecewise constant continuous functions from $X$ to $R$. We recall (see, e.g., \cite[Example~2.4.3]{BMO13a}) that a continuous function $f:X\to R$ is piecewise constant if there exist a clopen partition $\{P_1,\dots,P_n\}$ of $X$ and $a_i\in R$ such that $f(x)=a_i$ for each $x\in P_i$. Let $PC(X,R)$ be the subset of $C(X,R)$ consisting of piecewise constant functions. Then it is obvious that $PC(X,R)$ is an $R$-subalgebra of $C(X,R)$, and it follows from the definitions of $FC(X,R)$ and $PC(X,R)$ that $FC(X,R)=PC(X,R)$. By Proposition~\ref{prop:3.4}, $PC(X,R)=C(X,R_{\func{disc}})$. Thus, when $R$ is totally ordered, another way to think about the Boolean power of $R$ by $B$ is as the $R$-algebra of piecewise constant continuous functions from the Stone space $X$ of $B$ to $R$, where $R$ has the interval topology. Consequently, for a totally ordered  ring $R$, we obtain the following two representations of a Specker $R$-algebra: as the $R$-algebra $C(X, R_{\func{disc}})$ or as the $R$-algebra $PC(X,R)$. In \cite[Sec.~5]{BMO13a} it is proved that a Specker $\mathbb R$-algebra is isomorphic to $PC(X,\mathbb R)$. As follows from the discussion above, this result is a particular case of Theorem~\ref{new def}.
\end{remark}

\bibliographystyle{amsplain}
\bibliography{Gelfand}

\def\cprime{$'$}
\providecommand{\bysame}{\leavevmode\hbox to3em{\hrulefill}\thinspace}
\providecommand{\MR}{\relax\ifhmode\unskip\space\fi MR }
\providecommand{\MRhref}[2]{%
  \href{http://www.ams.org/mathscinet-getitem?mr=#1}{#2}
}
\providecommand{\href}[2]{#2}
\begin{thebibliography}{10}

\bibitem{BB67}
B.~Banaschewski and G.~Bruns, \emph{Injective hulls in the category of
  distributive lattices}, J. Reine Angew. Math. \textbf{232} (1968), 102--109.

\bibitem{BN80}
B.~Banaschewski and E.~Nelson, \emph{Boolean powers as algebras of continuous
  functions}, Dissertationes Math. (Rozprawy Mat.) \textbf{179} (1980), 51.

\bibitem{Ber72}
G.~M. Bergman, \emph{Boolean rings of projection maps}, J. London Math. Soc.
  \textbf{4} (1972), 593--598.

\bibitem{BMMO13b}
G.~Bezhanishvili, V.~Marra, P.~J. Morandi, and B.~Olberding, \emph{Proximities
  on {S}pecker algebras and de {V}ries powers}, submitted, 2013.

\bibitem{BMO13a}
G.~Bezhanishvili, P.~J. Morandi, and B.~Olberding, \emph{Bounded {A}rchimedean
  $\ell$-algebras and {G}elfand-{N}eumark-{S}tone duality}, Theory and
  Applications of Categories, 2013, to appear.

\bibitem{Bir79}
G.~Birkhoff, \emph{Lattice theory}, third ed., American Mathematical Society
  Colloquium Publications, vol.~25, American Mathematical Society, Providence,
  R.I., 1979.

\bibitem{BS81}
S.~Burris and H.~P. Sankappanavar, \emph{A course in universal algebra},
  Graduate Texts in Mathematics, vol.~78, Springer-Verlag, New York, 1981.

\bibitem{Con74}
P.~Conrad, \emph{Epi-archimedean groups}, Czechoslovak Math. J. \textbf{24
  (99)} (1974), 192--218.

\bibitem{Fos53a}
A.~L. Foster, \emph{Generalized ``{B}oolean'' theory of universal algebras.
  {I}. {S}ubdirect sums and normal representation theorem}, Math. Z.
  \textbf{58} (1953), 306--336.

\bibitem{Fos53b}
\bysame, \emph{Generalized ``{B}oolean'' theory of universal algebras. {II}.
  {I}dentities and subdirect sums of functionally complete algebras}, Math. Z.
  \textbf{59} (1953), 191--199.

\bibitem{Fos61}
\bysame, \emph{Functional completeness in the small. {A}lgebraic structure
  theorems and identities}, Math. Ann. \textbf{143} (1961), 29--58.

\bibitem{GJ60}
L.~Gillman and M.~Jerison, \emph{Rings of continuous functions}, The University
  Series in Higher Mathematics, D. Van Nostrand Co., Inc., Princeton,
  N.J.-Toronto-London-New York, 1960.

\bibitem{HJ65}
M.~Henriksen and M.~Jerison, \emph{The space of minimal prime ideals of a
  commutative ring}, Trans. Amer. Math. Soc. \textbf{115} (1965), 110--130.

\bibitem{Mac71}
S.~MacLane, \emph{Categories for the working mathematician}, Springer-Verlag,
  New York, 1971, Graduate Texts in Mathematics, Vol. 5.

\bibitem{Rib69}
P.~Ribenboim, \emph{Boolean powers}, Fund. Math. \textbf{65} (1969), 243--268.

\bibitem{Rot73}
G.-C. Rota, \emph{The valuation ring of a distributive lattice}, Proceedings of
  the {U}niversity of {H}ouston {L}attice {T}heory {C}onference ({H}ouston,
  {T}ex., 1973), Dept. Math., Univ. Houston, Houston, Tex., 1973, pp.~574--628.

\end{thebibliography}

\bigskip

\noindent Department of Mathematical Sciences, New Mexico State University, Las Cruces NM 88003-8001

\bigskip

\noindent Dipartimento di Matematica ``Federigo Enriques," Universit\`a degli Studi di Milano, via Cesare Saldini 50, I-20133 Milano, Italy

\bigskip

\noindent gbezhani@nmsu.edu, vincenzo.marra@unimi.it, pmorandi@nmsu.edu, oberdin@nmsu.edu

\end{document}